\documentclass[11pt,reqno,sumlimits]{amsart}
\usepackage{amsfonts, amsmath, amscd, amssymb, euscript, amsthm, array, booktabs,
dcolumn, shortvrb, tabularx, units, url, mathrsfs}
\usepackage[pdftex]{graphicx}
\usepackage[all]{xy}
\normalsize
\DeclareMathOperator{\B}{B}
\DeclareMathOperator{\op}{op}
\DeclareMathOperator{\WF}{WF}
\DeclareMathOperator{\BS}{BS}

\DeclareMathOperator{\pro}{prod}
\DeclareMathOperator{\mo}{mod}

\DeclareMathOperator{\LL}{L}

\DeclareMathOperator{\topo}{t}
\DeclareMathOperator{\todd}{Todd}
\DeclareMathOperator{\an}{a}
\DeclareMathOperator{\FL}{FL}

\DeclareMathOperator{\BU}{BU}

\DeclareMathOperator{\ind}{ind}

\DeclareMathOperator{\spin}{spin}

\DeclareMathOperator{\CS}{CS}

\DeclareMathOperator{\odd}{odd}
\DeclareMathOperator{\even}{even}

\DeclareMathOperator{\im}{Im}

\DeclareMathOperator{\ch}{ch}

\DeclareMathOperator{\rk}{rank}

\begin{document}
\setlength{\baselineskip}{1.5\baselineskip}
\theoremstyle{definition}
\newtheorem{coro}{Corollary}
\newtheorem*{note}{Note}
\newtheorem{thm}{Theorem}
\newtheorem*{ax}{Axiom}
\newtheorem{defi}{Definition}
\newtheorem{lemma}{Lemma}
\newtheorem*{claim}{Claim}
\newtheorem{exam}{Example}
\newtheorem{prop}{Proposition}
\newtheorem{remark}{Remark}
\newcommand{\wt}[1]{{\widetilde{#1}}}
\newcommand{\ov}[1]{{\overline{#1}}}
\newcommand{\wh}[1]{{\widehat{#1}}}
\newcommand{\poin}{Poincar$\acute{\textrm{e }}$}
\newcommand{\deff}[1]{{\bf\emph{#1}}}
\newcommand{\boo}[1]{\boldsymbol{#1}}
\newcommand{\abs}[1]{\lvert#1\rvert}
\newcommand{\norm}[1]{\lVert#1\rVert}
\newcommand{\inner}[1]{\langle#1\rangle}
\newcommand{\poisson}[1]{\{#1\}}
\newcommand{\biginner}[1]{\Big\langle#1\Big\rangle}
\newcommand{\set}[1]{\{#1\}}
\newcommand{\Bigset}[1]{\Big\{#1\Big\}}
\newcommand{\BBigset}[1]{\bigg\{#1\bigg\}}
\newcommand{\dis}[1]{$\displaystyle#1$}
\newcommand{\R}{\mathbb{R}}
\newcommand{\EE}{\mathbb{E}}
\newcommand{\N}{\mathbb{N}}
\newcommand{\Z}{\mathbb{Z}}
\newcommand{\Q}{\mathbb{Q}}
\newcommand{\E}{\mathcal{E}}
\newcommand{\EEE}{\mathscr{E}}
\newcommand{\G}{\mathcal{G}}
\newcommand{\F}{\mathcal{F}}
\newcommand{\V}{\mathcal{V}}
\newcommand{\W}{\mathcal{W}}
\newcommand{\SSS}{\mathcal{S}}
\newcommand{\h}{\mathbb{H}}
\newcommand{\g}{\mathfrak{g}}
\newcommand{\C}{\mathbb{C}}
\newcommand{\A}{\mathcal{A}}
\newcommand{\MM}{\mathcal{M}}
\newcommand{\HH}{\mathcal{H}}
\newcommand{\D}{\mathcal{D}}
\newcommand{\PP}{\mathcal{P}}
\newcommand{\M}{\mathcal{M}}
\newcommand{\KK}{\mathbb{K}}
\newcommand{\RRR}{\mathscr{R}}
\newcommand{\AAA}{\mathscr{A}}
\newcommand{\DDD}{\mathscr{D}}
\newcommand{\so}{\mathfrak{so}}
\newcommand{\gl}{\mathfrak{gl}}
\newcommand{\aaa}{\mathbb{A}}
\newcommand{\bbb}{\mathbb{B}}
\newcommand{\DD}{\mathsf{D}}
\newcommand{\RR}{\mathsf{R}}
\newcommand{\FF}{\mathbb{F}}
\newcommand{\GG}{\mathbb{G}}
\newcommand{\ccc}{\bold{c}}
\newcommand{\sss}{\mathbb{S}}
\newcommand{\cdd}[1]{\[\begin{CD}#1\end{CD}\]}
\normalsize
\title{Remarks on flat and differential $K$-theory}
\author[M.-H. Ho]{Man-Ho Ho}
\address{Department of Mathematics\\ Hong Kong Baptist University}
\email{homanho@math.hkbu.edu.hk}
\maketitle
\nocite{*}
\begin{abstract}
In this note we prove some results in flat and differential $K$-theory. The first one is
a proof of the compatibility of the differential topological index and the flat topological
index by a direct computation. The second one is the explicit isomorphisms between
Bunke-Schick differential $K$-theory and Freed-Lott differential $K$-theory.
\end{abstract}
\tableofcontents
\section{Introduction}
In this note we prove some results in flat and differential $K$-theory.  While some of
these results are known to the experts, the proofs given here have not appeared in the
literature. We first prove the compatibility of the flat topological index
$\ind^{\topo}_{\LL}$ and the differential topological index $\ind^{\topo}_{\FL}$
by a direct computation, i.e., the following diagram commutes (\cite[Proposition 8.10]{FL10})
\begin{equation}\label{eq 1}
\begin{CD}
K^{-1}_{\LL}(X; \R/\Z) @>i>> \wh{K}_{\FL}(X) \\ @V\ind^{\topo}VV @VV\ind^{\topo}_{\FL}V
\\ K^{-1}_{\LL}(B; \R/\Z) @>>i> \wh{K}_{\FL}(B)
\end{CD}
\end{equation}
where $i$ is the canonical inclusion, $K^{-1}_{\LL}(X; \R/\Z)$ is the geometric model of
$K$-theory with $\R/\Z$ coefficients and $\wh{K}_{\FL}(X)$ is Freed-Lott differential
$K$-theory. The commutativity of (\ref{eq 1}) is a consequence of the compatibility of
the differential analytic index $\ind^{\an}_{\FL}$ and the flat analytic index
$\ind^{\an}_{\LL}$ together with the differential family index theorem
\cite[Theorem 7.35]{FL10}. The differential topological index $\ind^{\topo}_{\FL}$ is
defined to be the composition of an embedding pushforward and a projection pushforward.
When defining the embedding pushforward, currential $K$-theory \cite[\S 2.28]{FL10} is used
instead of differential $K$-theory due to the Bismut-Zhang current \cite[Definition 1.3]{BZ93}.
It is not clear whether currential $K$-theory should be regarded as a differential cohomology
or a ``differential homology" (see \cite[\S 4.5]{BS10} for a detailed discussion), so it may
be clearer by looking at the direct computation.

Second we construct the unique natural isomorphisms between Bunke-Schick differential
$K$-theory \cite{BS09} and Freed-Lott differential $K$-theory by writing down the explicit
formulas, which are inspired by \cite[Corollary 5.5]{BS09}. The uniqueness follows from
\cite[Theorem 3.10]{BS10a}. Together with \cite[Theorem 4.34]{R08} and \cite[Theorem 1]{H12}
all the explicit isomorphisms between all the existing differential $K$-groups \cite{HS05},
\cite{BS09}, \cite{FL10}, \cite{SS10} are known.

The paper is organized as follows: Section 2 contains all the necessary background
material, including the Freed-Lott differential $K$-theory, the differential topological
index, the pairing between flat $K$-theory and $K$-homology, and Bunke-Schick differential
$K$-theory. In Section 3 we prove the main results.

\section*{Acknowledgement}

The author would like to thank Steve Rosenberg for valuable comments and suggestions, and
Thomas Schick for his comments on the explicit isomorphisms between Bunke-Schick
differential $K$-theory and Freed-Lott differential $K$-theory.

\section{Background material}
\subsection{Freed-Lott differential $K$-theory and the differential topological index}

In this section we review Freed-Lott differential $K$-theory and the construction of the
differential topological index \cite[\S 4, 5]{FL10}. We refer the readers to \cite{FL10}
for the details.

The Freed--Lott differential $K$-group $\wh{K}_{\FL}(X)$ is the abelian group generated
by quadruples $\E=(E, h, \nabla, \phi)$, where $(E,h,\nabla)\to X$ is a complex vector
bundle with a Hermitian metric $h$ and a unitary connection $\nabla$, and
\dis{\phi\in\frac{\Omega^{\odd}(X)}{\im(d)}}. The only relation is $\E_1=\E_2$ if and
only if there exists a generator $(F, h^F, \nabla^F, \phi^F)$ of $\wh{K}_{\FL}(X)$ such
that $E_1\oplus F\cong E_2\oplus F$ and $\phi_1-\phi_2=\CS(\nabla^{E_2}\oplus\nabla^F,
\nabla^{E_1}\oplus\nabla^F)$.

There is an exact sequence \cite[(2.20)]{FL10}
\begin{equation}\label{eq 3}
\begin{CD}
0 @>>> K^{-1}_{\LL}(X; \R/\Z) @>i>> \wh{K}_{\FL}(X) @>\ch_{\wh{K}}>> \Omega^{\even}_{\BU}(X)
@>>> 0
\end{CD}
\end{equation}
where $K^{-1}_{\LL}(X; \R/\Z)$ is the geometric model of $\R/\Z$ $K$-theory \cite{L94}, $i$
is the canonical inclusion map,
$$\Omega^{\even}_{\BU}(X)=\set{\omega\in\Omega^{\even}_{d=0}(X)|[\omega]\in\im(r\circ\ch:
K^0(X)\to H^{\even}(X; \R))},$$
and $\ch_{\wh{K}_{\FL}}(E, h, \nabla, \phi):=\ch(\nabla)+d\phi$. Elements in $K^{-1}_{\LL}
(X; \R/\Z)$ are required to have virtual rank zero. The canonical inclusion map $i$ in
(\ref{eq 3}) is defined by $i(E, h, \nabla, \phi)=(E, h, \nabla, \phi)$.

Let $X\to B$ and $Y\to B$ be fiber bundles of smooth manifolds with $X$ compact. Let
$g^{T^VX}$ and $g^{T^VY}$ be metrics on the vertical bundles $T^VX\to X$ and $T^VY\to Y$
respectively, and assume there are horizontal distributions $T^HX$ and $T^HY$. Let
$\E=(E, h^E, \nabla^E, \phi)\in\wh{K}_{\FL}(X)$ and $\iota:X\hookrightarrow Y$ be an
embedding of manifolds. We assume the codimension of $X$ in $Y$ is even, and the normal
bundle $\nu\to X$ of $X$ in $Y$ carries a $\spin^c$ structure. As in \cite[\S 5]{FL10}
we assume for each $b\in B$, the map $\iota_b:X_b\to Y_b$ is an isometric embedding
and is compatible with projections to $B$. Denote by $\SSS(\nu)\to X$ the spinor bundle
associated to the $\spin^c$-structure of $\nu\to X$. We can locally choose a $\spin$
stricture for $\nu\to X$ with spinor bundle $\SSS^{\spin}(\nu)$. Then there exists a
locally defined Hermitian line bundle $L^{\frac{1}{2}}(\nu)$ such that $\SSS(\nu)\cong
\SSS^{\spin}(\nu)\otimes L^{\frac{1}{2}}(\nu)$. Note that the tensor product on the right
is globally defined, and so is the Hermitian line bundle $L(\nu)\to X$ defined by
$L(\nu):=(L^{\frac{1}{2}}(\nu))^2$. Let $\nabla^\nu$ be a metric compatible connection on
$\nu\to X$. It has a unique lift to a connection on $\SSS^{\spin}(\nu)$, still denoted by
$\nabla^\nu$. Choose a unitary connection $\nabla^{L(\nu)}$ on $L(\nu)\to X$, which induces
a connection on $L^{\frac{1}{2}}(\nu)$. The tensor product of $\nabla^\nu$ and the induced
connection on $L^{\frac{1}{2}}(\nu)$ is a connection on $\SSS(\nu)\to X$, denoted by
$\wh{\nabla}^\nu$. Define
$$\todd(\wh{\nabla}^\nu):=\wh{A}(\nabla^\nu)\wedge e^{\frac{1}{2}c_1(\nabla^{L(\nu)})}.$$
The embedding pushforward $\wh{\iota}_*:\wh{K}_{\FL}(X)\to{}_\delta\wh{K}_{\FL}(Y)$
\cite[Definition 4.14]{FL10} is defined to be
$$\wh{\iota}_*(\E)=\bigg(F, h^F, \nabla^F, \frac{\phi^E}{\todd(\wh{\nabla}^\nu)}\wedge
\delta_X-\gamma\bigg),$$
where ${}_\delta\wh{K}_{\FL}(Y)$ is the currential $K$-group, $\delta_X$ is the current of
integration over $X$ and $\gamma$ is the Bismut-Zhang current. $(F, h^F, \nabla^F)$ is a
Hermitian bundle with a Hermitian metric and a unitary connection chosen as in
\cite[Lemma 4.4]{FL10}. Note that $\gamma$ satisfies the following transgression formula
\cite[Theorem 1.4]{BZ93}
$$d\gamma=\ch(\nabla^F)-\frac{\ch(\nabla^E)}{\todd(\wh{\nabla}^\nu)}\wedge\delta_X.$$
As noted in \cite[p.926]{FL10} the horizontal distributions of the fiber bundles $X\to B$
and $Y\to B$ need not be compatible. An odd form \dis{\wt{C}\in\frac{\Omega^{\odd}(X)}{\im
(d)}} is defined to correct this non-compatibility, and it satisfies the following
transgression formula \cite[(5.6)]{FL10}
$$d\wt{C}=\iota^*\todd(\wh{\nabla}^{T^VY})-\todd(\wh{\nabla}^{T^VX})\wedge\todd
(\wh{\nabla}^\nu).$$
The modified embedding pushforward $\wh{\iota}^{\mo}_*:\wh{K}_{\FL}(X)\to{}_{\WF}
\wh{K}_{\FL}(Y)$ \cite[Definition 5.8]{FL10} is defined to be
\begin{equation}\label{eq 13}
\wh{\iota}_*^{\mo}(\E):=\wh{\iota}_*(\E)-j\bigg(\frac{\wt{C}}{\iota^*\todd
(\wh{\nabla}^{T^VY})\wedge\todd(\wh{\nabla}^\nu)}\wedge\ch_{\wh{K}_{\FL}}(\E)\wedge
\delta_X\bigg).
\end{equation}
See \cite[\S 3.1]{FL10} for the definition of ${}_{\WF}\wh{K}_{\FL}(X)$.

The differential topological index $\ind^{\topo}_{\FL}:\wh{K}_{\FL}(X)\to\wh{K}_{\FL}
(B)$ \cite[Definition 5.34]{FL10} is defined by taking $Y=\sss^N\times B$ for some even
$N$ and composing the embedding pushforward with the submersion pushforward
$\wh{\pi}^{\pro}_*$ defined in \cite[Lemma 5.13]{FL10}, i.e., $\ind^{\topo}_{\FL}:=
\wh{\pi}^{\pro}_*\circ\wh{\iota}^{\mo}_*$.

\subsection{Pairing between flat $K$-theory and topological $K$-homology}

Let $X$ be an odd-dimensional closed $\spin^c$ manifold. Let $\E=(E, h^E, \nabla^E, \phi)
\in{}_\delta\wh{K}_{\FL}(X)$, and $\DD^{X, E}$ be the twisted Dirac operator on $\SSS(X)
\otimes E\to X$. A modified reduced eta-invariant $\bar{\eta}(X, \E)\in\R/\Z$
\cite[Definition 2.33]{FL10} is defined by
$$\bar{\eta}(X, \E):=\bar{\eta}(\DD^{X, E})+\int_X\todd(\wh{\nabla}^{TX})\wedge\phi
\mod\Z.$$
$\bar{\eta}:{}_\delta\wh{K}_{\FL}(X)\to\R/\Z$ is a well defined homomorphism
\cite[Prop 2.25]{FL10}. If $\E$ is a generator of $K^{-1}_{\LL}(X; \R/\Z)$, by
\cite[(2.37)]{FL10} we have
\begin{equation}\label{eq 15}
\bar{\eta}(X, i(\E))=\inner{[X], \E},
\end{equation}
where $[X]\in K_{-1}(X)$ is the fundamental $K$-homology class. Here $\inner{[X], \E}$
is the perfect pairing between flat $K$-theory and topological $K$-homology
\cite[Prop 3]{L94}
\begin{equation}\label{eq 16}
K^{-1}_{\LL}(X; \R/\Z)\times K_{-1}(X)\to\R/\Z.
\end{equation}

\subsection{Bunke-Schick differential $K$-theory}

In this subsection we briefly recall Bunke-Schick differential $K$-theory $\wh{K}_{\BS}$,
and refer to \cite{BS09} for the details.

A generator of $\wh{K}_{\BS}(B)$ is of the form $(\EEE, \phi)$, where $\EEE$ is an
even-dimensional geometric family \cite[Definition 2.2]{BS09} over a compact manifold $B$
and \dis{\phi\in\frac{\Omega^{\odd}(B)}{\im(d)}}. Roughly speaking a geometric family over
$B$ is the geometric data needed to construct the index bundle. There is a well defined
notion of isomorphic and sum of generators \cite[Definition 2.5, 2.6]{BS09}. Two geometric
families $(\EEE_0, \phi_0)$ and $(\EEE_1, \phi_1)$ are equivalent if there exists a geometric
family $(\EEE', \phi')$ such that $(\EEE_0, \rho_0)+(\EEE', \phi')$ is paired with
$(\EEE_1, \rho_1)+(\EEE', \phi')$ \cite[Definition 2.10, Lemma 2.13]{BS09}. Two generators
$(\EEE_0, \phi_0)$ and $(\EEE_1, \phi_1)$ are paired if
$$\rho_1-\rho_0=\eta^{\B}((\EEE_0\sqcup_B(\EEE_1)^{\op})_t),\footnote{It differs by a sign in
\cite{BS09}.}$$
where $(\EEE\sqcup_B(\EEE')^{\op})_t$ is a certain tamed geometric family
\cite[Definition 2.7]{BS09}, and $\eta^{\B}$ is the Bunke eta form \cite{B09}.

As noted in \cite[2.14]{BS09} and \cite[4.2.1]{B09}, a complex vector bundle $E\to B$ with
a Hermitian metric $h^E$ and a unitary connection $\nabla^E$ can be naturally considered
as a zero-dimensional geometric family over $B$, denoted by $\EE$.

\section{Main results}
\subsection{Compatibility of the topological indices}

Note that every element $\E-\F\in\wh{K}_{\FL}(X)$ can be written in the form
$$\wt{\E}-[n].$$
Here $\wt{\E}=(E\oplus G, h^E\oplus h^G, \nabla^E\oplus\nabla^G, \phi^E+\phi^G)$, where
$(G, h^G, \nabla^G, \phi^G)$ is a generator of $\wh{K}_{\FL}(X)$ such that
$$(F\oplus G, h^F\oplus h^G, \nabla^F\oplus\nabla^G, \phi^F+\phi^G)=(\C^n, h, d, 0)=:[n].$$
The existence of the connection $\nabla^G$ such that $\CS(\nabla^F\oplus\nabla^G, d)=0$,
where $d$ is the trivial connection on the trivial bundle $X\times\C^n\to X$, follows
from \cite[Theorem 1.8]{SS10}. Here $\phi^G:=-\phi^F$. Henceforth we assume an element of
$\wh{K}_{\FL}(X; \R/\Z)$ is of the form $\E-[n]$. These arguments also apply to elements
in $K^{-1}_{\LL}(X; \R/\Z)$.
\begin{prop}\label{prop 5}
Let $\pi:X\to B$ be a fiber bundle with $X$ compact and such that the fibers are of even
dimension. The following diagram commutes.
\cdd{K^{-1}_{\LL}(X; \R/\Z) @>i>> \wh{K}_{\FL}(X) \\ @V\ind^{\topo}VV @VV\ind^{\topo}_{\FL}V
\\ K^{-1}_{\LL}(B; \R/\Z) @>>i> \wh{K}_{\FL}(B)}
\end{prop}
\begin{proof}
Let $\E'-[n]'\in K^{-1}_{\LL}(X)$ and write $\E-[n]=i(\E'-[n]')$, where $i$ is given in
(\ref{eq 3}). Consider the difference
$$h:=\ind^{\topo}_{\FL}(\E-[n])-i(\ind^{\topo}(\E'-[n]')).$$
We prove that $h=0$. By \cite[Lemma 5.36]{FL10} and the fact that $\ch_{\wh{K}_{\FL}}\circ
i=0$ (see (\ref{eq 3})), we have
\begin{displaymath}
\begin{split}
&\qquad\ch_{\wh{K}_{\FL}}(\ind^{\topo}_{\FL}(\E-[n]))-\ch_{\wh{K}_{\FL}}(i
(\ind^{\topo}(\E'-[n]')))\\
&=\ch_{\wh{K}_{\FL}}(\ind^{\topo}_{\FL}(\E-[n]))\\
&=\int_{X/B}\todd(\wh{\nabla}^{T^VX})\wedge(\ch(\nabla^E)-\rk(E)+d\phi^E)\\
&=0.
\end{split}
\end{displaymath}
By (\ref{eq 3}), there exists an element $a\in K^{-1}(B; \R/\Z)$ such that $i(a)=h$. To
prove $a=0\in K^{-1}_{\LL}(B; \R/\Z)$, it follows from (\ref{eq 16}) that it is sufficient
to show that for all $\alpha\in K_{-1}(B; \Z)$,
\begin{equation}\label{eq 29}
\inner{\alpha, a}=0\in\R/\Z.
\end{equation}
Using the geometric picture of $K$-homology \cite{BHS07}, we may, without loss of generality,
let $\alpha=f_*[M]$ for some smooth map $f:M\to B$, where $M$ is a closed odd-dimensional
$\spin^c$ manifold, and $[M]$ is the fundamental $K$-homology in $K_{-1}(M)$. Since
$\inner{\alpha, a}=\inner{[M], f^*a}$, we pull everything back to $M$ and we may assume
$B$ is an arbitrary closed odd-dimensional $\spin^c$ manifold. Thus proving (\ref{eq 29})
is equivalent to proving
\begin{equation}\label{eq 30}
\inner{[B], a}=0\in\R/\Z.
\end{equation}
Since
$$\inner{[B], a}=\bar{\eta}(B, \ind^{\topo}_{\FL}(\E-[n]))-\bar{\eta}(B, i(\ind^{\topo}
(\E'-[n]')))\mod\Z,$$
proving (\ref{eq 30}) is equivalent to proving
\begin{equation}\label{eq 31}
\bar{\eta}(B, \ind^{\topo}_{\FL}(\E-[n]))=\bar{\eta}(B, i(\ind^{\topo}(\E'-[n]')))
\mod\Z.
\end{equation}
In the following, we write $a\equiv b$ as $a=b\mod\Z$. By \cite[(6.7)]{FL10}, we have
\begin{equation}\label{eq 32}
\begin{split}
\bar{\eta}(B, \ind^{\topo}_{\FL}(\E-[n]))&\equiv\bar{\eta}(\DD^{X, E-n})+\int_X
\frac{\iota^*\todd(\wh{\nabla}^{T(\sss^N\times B)})}{\todd(\wh{\nabla}^\nu)}\wedge
\phi^E\\
&\qquad-\int_X\frac{\pi^*\todd(\wh{\nabla}^{TB})}{\todd(\wh{\nabla}^\nu)}\wedge\wt{C}
\wedge\ch_{\wh{K}_{\FL}}(\E-[n])\\
&\equiv\bar{\eta}(\DD^{X, E-n})+\int_X\frac{\iota^*\todd(\wh{\nabla}^{T(\sss^N\times
B)})}{\todd(\wh{\nabla}^\nu)}\wedge\phi^E
\end{split}
\end{equation}
as $\ch_{\wh{K}_{\FL}}(\E-[n])=\ch_{\wh{K}_{\FL}}(i(\E'-[n]'))=0$. On the other hand, by
\cite[(49)]{L94}, we have
\begin{equation}\label{eq 33}
\begin{split}
\bar{\eta}(B, i(\ind^{\topo}(\E'-[n]')))&\equiv\inner{[B], \ind^{\topo}
(\E'-[n]')}\\
&=\inner{\pi^![B], \E-[n]}=\inner{[X], \E-[n]}=\bar{\eta}(X, \E-[n])\\
&\equiv\bar{\eta}(\DD^{X, E-n})+\int_X\todd(\wh{\nabla}^{TX})\wedge\phi^E.
\end{split}
\end{equation}
From (\ref{eq 32}) and (\ref{eq 33}) we have
\begin{equation}\label{eq 34}
\begin{split}
&\qquad\bar{\eta}(B, \ind^{\topo}_{\FL}(\E-[n]))-\bar{\eta}(B, i(\ind^{\topo}
(\E'-[n]')))\\
&\equiv\int_X\bigg(\frac{\iota^*\todd(\wh{\nabla}^{T(\sss^N\times
B)})}{\todd(\wh{\nabla}^\nu)}-\todd(\wh{\nabla}^{TX})\bigg)\wedge\phi^E\\
&\equiv\int_X\bigg(\frac{\iota^*\todd(\wh{\nabla}^{T(\sss^N\times
B)})-\todd(\wh{\nabla}^{TX})\wedge\todd(\wh{\nabla}^\nu)}{\todd(\wh{\nabla}^\nu)}\bigg)
\wedge\phi^E.
\end{split}
\end{equation}
Since $\ch_{\wh{K}_{\FL}}(\E-[n])=0$, it follows from (\ref{eq 13}) that
\begin{equation}\label{eq 99}
\wh{\iota}^{\mo}_*(\E-[n])=\wh{\iota}_*(\E-[n]).
\end{equation}
Recall that the purpose of the modified embedding pushforward $\wh{\iota}^{\mo}_*$ is to
correct the non-compatibility of the horizontal distributions $T^H(\sss^N\times B)$
and $T^HX$. By (\ref{eq 99}) we may assume that the horizontal distributions $T^H(\sss^N
\times B)$ and $T^HX$ are compatible by changing the one for $X$ to be the restriction of
the one for $\sss^N\times B$. Thus
$$\iota^*\todd(\wh{\nabla}^{T(\sss^N\times B)})=\todd(\wh{\nabla}^{TX})\wedge\todd
(\wh{\nabla}^\nu),$$
which implies that (\ref{eq 34}) is zero, and therefore $h=0$.
\end{proof}

\subsection{Explicit isomorphisms between $\wh{K}_{\BS}$ and $\wh{K}_{\FL}$}

In this subsection we construct the explicit isomorphisms between Bunke-Schick differential
$K$-group and the Freed-Lott differential $K$-group.
\begin{prop}
Let $B$ be a compact manifold. Define two maps $f:\wh{K}_{\FL}(B)\to\wh{K}_{\BS}(B)$ and
$g:\wh{K}_{\BS}(B)\to\wh{K}_{\FL}(B)$ by
\begin{displaymath}
\begin{split}
f(E, h, \nabla, \phi)&=[\EE, \phi],\\
g([\EEE, \phi])&=(\ind^{\an}(\EEE), h^{\ind^{\an}(\EEE)}, \nabla^{\ind^{\an}(\EEE)}, \phi),
\end{split}
\end{displaymath}
where, in the definition of $f$, $\EE$ is the zero-dimensional geometric family associated to
$(E, h, \nabla)$. Then $f$ and $g$ are well defined ring isomorphisms and are inverses to
each other.
\end{prop}
\begin{proof}
Note that it suffices to prove the statement under the assumption that $\ind^a(\EEE)\to B$
is actually given by a kernel bundle $\ker(\DD^E)\to B$ in the definition of $g$. Indeed,
by a standard perturbation argument every class in $\wh{K}_{\BS}$ has a representative
where this is satisfied.

First of all we prove that $f$ is well defined. Suppose
$$(E, h^E, \nabla^E, \phi^E)=(F, h^F, \nabla^F, \phi^F)\in\wh{K}_{\FL}(B).$$
Then there exists a generator $(G, h^G, \nabla^G, \phi^G)$ of $\wh{K}_{\FL}(B)$ such
that
\begin{equation}\label{eq 35}
\begin{split}
E\oplus G&\cong F\oplus G,\\
\phi^F-\phi^E&=\CS(\nabla^E\oplus\nabla^G, \nabla^F\oplus\nabla^G).
\end{split}
\end{equation}
Denote by $\FF$ and $\GG$ the zero-dimensional geometric families associated to $(F, h^F,
\nabla^F)$ and $(G, h^G, \nabla^G)$, respectively. We prove that $[\EE, \phi^E]=[\FF,
\phi^F]\in\wh{K}_{\BS}(B)$. Indeed, we prove that $(\EE+\GG, \phi^E+\phi^G)$ is paired
with $(\FF+\GG, \phi^F+\phi^G)$. We need to show $\EE\sqcup_B\GG\cong\FF\sqcup_B\GG$ and
\begin{equation}\label{eq 36}
(\phi^F+\phi^G)-(\phi^E+\phi^G)=\eta^{\B}(((\EE\sqcup_B\GG)\sqcup_B(\FF\sqcup_B
\GG)^{\op})_t)
\end{equation}
if such a taming exists. In the case of zero-dimensional geometric family, $\EE\sqcup_B\GG
\cong E\oplus G$ as vector bundles over $B$. Thus the first equality (\ref{eq 35}) implies
$\EE\sqcup_B\GG\cong\FF\sqcup_B\GG$. Since the underlying proper submersion of the trivial
geometric family is the identity map, the corresponding kernel bundle is just $E\to B$ by
the remark of \cite[Definition 4.7]{BS10}. Thus the taming in (\ref{eq 36}) exists and the
definition of $\eta^{\B}$ shows that
\begin{equation}\label{eq 37}
\eta^{\B}(((\EE\sqcup_B\GG)\sqcup_B(\FF\sqcup_B\GG)^{\op})_t)=\CS(\nabla^E\oplus\nabla^G,
\nabla^F\oplus\nabla^G).
\end{equation}
From (\ref{eq 35}) and (\ref{eq 36}) we see that $(\EE+\GG, \phi^E+\phi^G)$ is paired with
$(\FF+\GG, \phi^F+\phi^G)$. Thus $f$ is well defined.

For the map $g$, note that under our assumption we have $[\EEE, 0]=[\KK, \wt{\eta}(\EEE)]$
by \cite[Corollary 5.5]{BS09}, where $\KK$ is the trivial geometric family associated to
$(\ker(\DD^E), h^{\ker(\DD^E)}, \nabla^{\ker(\DD^E)})$ and $\wt{\eta}(\EEE)$ is the associated
Bismut-Cheeger eta form. Since $[\EEE, \phi]=[\KK, \wt{\eta}(\EEE)+\phi]$, $g$ can be written
as
$$g([\EEE, \phi])=g([\KK, \wt{\eta}(\EEE)+\phi])=(\ker(\DD^E), h^{\ker(\DD^E)},
\nabla^{\ker(\DD^E)}, \wt{\eta}(\EEE)+\phi).$$
We prove that $g$ is well defined. Suppose $[\EEE_1, \phi^1]=[\EEE_2, \phi^2]\in\wh{K}_{\BS}(B)$.
Since $[\EEE_i, \phi^i]=[\KK_i, \wt{\eta}(\EEE_i)+\phi^i]$ for $i=1, 2$, to prove $g([\EEE_1,
\phi^1])=g([\EEE_2, \phi^2])$ it suffices to show
\begin{equation}\label{eq 38}
\begin{split}
&~~~~~~(\ker(\DD^{E^1}), h^{\ker(\DD^{E^1})}, \nabla^{\ker(\DD^{E^1})}, \wt{\eta}(\EEE_1)+
\phi^1)\\
&=(\ker(\DD^{E^2}), h^{\ker(\DD^{E^2})}, \nabla^{\ker(\DD^{E^2})}, \wt{\eta}(\EEE_2)+\phi^2).
\end{split}
\end{equation}
Since $[\EEE_1, \phi^1]=[\EEE_2, \phi^2]$, there exists a taming $(\EEE_1\sqcup_B(\EEE_2)^{\op})_t$,
and therefore $\ker(\DD^{E^1})=\ker(\DD^{E^2})\in K(B)$. Thus it suffices to show
\begin{equation}\label{eq 39}
\CS(\nabla^{\ker(\DD^{E^2})}, \nabla^{\ker(\DD^{E^1})})=\wt{\eta}(\EEE_1)-\wt{\eta}(\EEE_2)+
\phi^2-\phi^1\in\frac{\Omega^{\odd}(B)}{\Omega^{\odd}_{\BU}(B)}
\end{equation}
by the exactness of \cite[(2.21)]{FL10}. Since
$$[\KK_1, \wt{\eta}(\EEE_1)+\phi^1]=[\EEE_1, \phi^1]=[\EEE_2, \phi^2]=[\KK_2, \wt{\eta}(\EEE_2)
+\phi^2],$$
it follows that there exists a taming $(\KK_2\sqcup_B(\KK_1)^{\op})_t$ such that
\begin{equation}\label{eq 40}
\wt{\eta}(\EEE_1)-\wt{\eta}(\EEE_2)+\phi^1-\phi^2=\eta^{\B}((\KK_2\sqcup_B(\KK_1)^{\op})_t).
\end{equation}
By the same reason as in (\ref{eq 37}) we have
\begin{equation}\label{eq 41}
\eta^{\B}((\KK_2\sqcup_B(\KK_1)^{\op})_t)=\CS(\nabla^{\ker(\DD^{E^2})},
\nabla^{\ker(\DD^{E^1})}).
\end{equation}
(\ref{eq 39}) follows by comparing (\ref{eq 40}) and (\ref{eq 41}). Thus $g$ is well defined.

We prove that $f$ and $g$ are inverses to each other. Let $(E, h, \nabla, \phi)$ be a generator
of $\wh{K}_{\FL}(B)$. Then
$$(g\circ f)(E, h, \nabla, \phi)=g[(\EE, \phi)]=(E, h, \nabla, \phi).$$
On the other hand, for a generator $(\EEE, \phi)$ of $\wh{K}_{\BS}(B)$,
\begin{displaymath}
\begin{split}
(f\circ g)([\EEE, \phi])&=f(\ker(\DD^E), h^{\ker(\DD^E)}, \nabla^{\ker(\DD^E)}, \wt{\eta}(\EEE)
+\phi)\\
&=[\KK, \wt{\eta}(\E)+\phi]\\
&=[\EEE, \phi]
\end{split}
\end{displaymath}
by \cite[Corollary 5.5]{BS09} again.

Since $f$ is a ring homomorphism, the same is true for $g$. Thus $f$ and $g$ are ring
isomorphisms and are inverses to each other.
\end{proof}
\bibliographystyle{amsplain}
\bibliography{MBib}

\providecommand{\bysame}{\leavevmode\hbox to3em{\hrulefill}\thinspace}
\providecommand{\MR}{\relax\ifhmode\unskip\space\fi MR }
\providecommand{\MRhref}[2]{%
  \href{http://www.ams.org/mathscinet-getitem?mr=#1}{#2}
}
\providecommand{\href}[2]{#2}
\begin{thebibliography}{10}

\bibitem{BHS07}
P.~Baum, N.~Higson, and T.~Schick, \emph{On the equivalence of geometric and
  analytic {$K$}-homology}, Pure Appl. Math. Q. \textbf{3} (2007), no.~1, part
  3, 1--24.

\bibitem{BZ93}
J.M. Bismut and W.~Zhang, \emph{Real embeddings and eta invariants}, Math. Ann.
  \textbf{295} (1993), no.~4, 661--684.

\bibitem{B09}
U.~Bunke, \emph{Index theory, eta forms, and {D}eligne cohomology}, Mem. Amer.
  Math. Soc. \textbf{198} (2009), no.~928, vi+120.

\bibitem{BS09}
U.~Bunke and T.~Schick, \emph{Smooth {$K$}-theory}, Ast\'erisque, no.~328,
  45--135 (2010) (English, with English and French summaries).

\bibitem{BS10a}
\bysame, \emph{Uniqueness of smooth extensions of generalized cohomology
  theories}, J. Topol. \textbf{3} (2010), no.~1, 110--156.

\bibitem{BS10}
\bysame, \emph{Differential {$K$}-theory. {A} survey}, Global Differential
  Geometry (Berlin Heidelberg) (C.~B$\ddot{\textrm{a}}$r, J.~Lohkamp, and
  M.~Schwarz, eds.), Springer Proceedings in Mathematics, vol.~17,
  Springer-Verlag, 2012, pp.~303--358.

\bibitem{FL10}
D.~Freed and J.~Lott, \emph{An index theorem in differential ${K}$-theory},
  Geom. Topol. \textbf{14} (2010), no.~2, 903--966.

\bibitem{H12}
M.-H. Ho, \emph{The differential analytic index in {S}imons-{S}ullivan
  differential {$K$}-theory}, Ann. Global Anal. Geom. \textbf{42} (2012),
  no.~4, 523--535.

\bibitem{HS05}
M.~J. Hopkins and I.~M. Singer, \emph{Quadratic functions in geometry,
  topology, and {$M$}-theory}, J. Differential Geom. \textbf{70} (2005), no.~3,
  329--452.

\bibitem{R08}
K.~Klonoff, \emph{An index theorem in differential {$K$}-theory}, Ph.D. thesis,
  The University of Texas at Austin, 2008, p.~119.

\bibitem{L94}
J.~Lott, \emph{$\mathbb{R}/\mathbb{Z}$ index theory}, Comm. Anal. Geom.
  \textbf{2} (1994), no.~2, 279--311.

\bibitem{SS10}
J.~Simons and D.~Sullivan, \emph{Structured vector bundles define differential
  {$K$}-theory}, Quanta of maths (Providence, RI), Clay Math. Proc., vol.~11,
  Amer. Math. Soc., 2010, pp.~579--599.

\end{thebibliography}
\end{document}